\documentclass[12pt]{article}
\usepackage{amsmath,amsthm,amssymb,enumerate,mathscinet}
\usepackage{fullpage}
\usepackage{graphicx}

\newtheorem{theorem}{Theorem}[section]

\newtheorem{lemma}[theorem]{Lemma}

\newtheorem{claim}[theorem]{Claim}

\newtheorem{thm}[theorem]{Theorem}

\newtheorem{subclaim}[theorem]{Subclaim}

\newcommand{\msc}[1]{\begin{center}MSC2000: #1.\end{center}}

\numberwithin{equation}{section}

\def\eps{\varepsilon}

\def\COMMENT#1{}

\title{Tilings in randomly perturbed dense graphs} 

\author{J\'ozsef Balogh\footnote{Department of Mathematical Sciences,
 University of Illinois at Urbana-Champaign, Urbana, Illinois 61801, USA {\tt
jobal@math.uiuc.edu}.  Research is partially supported by NSF Grant DMS-1500121 and Arnold O. Beckman Research Award (UIUC Campus Research Board 15006).}
, Andrew Treglown\footnote{University of Birmingham, United Kingdom, {\tt a.c.treglown@bham.ac.uk}. Research supported by EPSRC grant EP/M016641/1.}
\ and Adam Zsolt Wagner\footnote{University of Illinois at Urbana-Champaign, Urbana, Illinois 61801, USA, {\tt
zawagne2@illinois.edu}. }}

\begin{document}
 \maketitle

\begin{abstract}
A perfect $H$-tiling in a graph $G$ is a collection of vertex-disjoint copies of a graph $H$ in $G$ that together cover all the vertices in $G$. In this paper we investigate perfect $H$-tilings in a random graph model
introduced by Bohman, Frieze and Martin~\cite{bfm1} in which one starts with a dense graph and then adds $m$ random edges to it.
Specifically, for \emph{any} fixed graph $H$, we determine the number of random edges required to add to an arbitrary graph of linear minimum degree in order to ensure the resulting graph contains a perfect $H$-tiling with high probability. Our proof utilises Szemer\'edi's Regularity lemma~\cite{szemeredi} as well as a special case of a result of Koml\'os~\cite{komlos} concerning almost perfect $H$-tilings in dense graphs.
\end{abstract}
 \msc{5C35, 5C70, 5C80}
\section{Introduction}
Embedding problems form a central part of both extremal and random graph theory. Indeed, many results in extremal graph theory concern minimum degree conditions that force some spanning substructure. For example,
a foundation stone in the subject is Dirac's theorem~\cite{dirac} from 1952 which states that every graph $G$ on $n \geq 3$ vertices and with minimum degree $\delta (G)\geq n/2$ is Hamiltonian. 
More recently, an important  paper of B\"ottcher, Schacht and Taraz~\cite{bot} resolved the Bollob\'as--Koml\'os conjecture; specifically, it provided a minimum degree condition which ensures a graph contains
every $r$-partite spanning subgraph of bounded degree and small bandwidth.

Recall that the Erd\H{o}s--R\'enyi random graph $G_{n,p}$ consists of vertex set $[n]:=\{1,\dots,n\}$ where each edge is present with probability $p$, independently of all other choices.
In this setting, a key question is to establish the \emph{threshold} at which $G_{n,p}$, with high probability, contains some spanning subgraph.
In the case of Hamilton cycles, P\'osa~\cite{posa1} showed that if $p\gg \log n /n$ then  asymptotically almost surely (a.a.s.) $G_{n,p}$ is Hamiltonian whilst  if $p\ll \log n /n$
a.a.s.~$G_{n,p}$ is not Hamiltonian. In general though few thresholds for embedding a fixed spanning subgraph $H$ in $G_{n,p}$ are known.

Bohman, Frieze and Martin~\cite{bfm1} introduced a model which in some sense connects the two aforementioned questions together. Indeed, in their model one starts with a dense graph and then 
adds $m$ random edges to it. A natural problem in this setting is to determine \emph{how many} random edges are required to ensure that the resulting graph  a.a.s.~contains a given graph $H$ as a spanning subgraph.
For example, the main result in~\cite{bfm1} states that for every $\alpha>0$, there is a $c=c(\alpha)$ such that if we start with an arbitrary $n$-vertex graph $G$ of minimum degree $\delta(G)\geq \alpha n$ and
add $cn$ random edges to it, then a.a.s.~the resulting graph is Hamiltonian. This result is best possible in the sense that there are graphs $G$ of linear minimum degree that require a linear number of edges to be added to become Hamiltonian (for example, consider any complete bipartite graph with vertex classes of size $an,bn$ where $0<a<b<1$ and $a+b=1$).
Recently, Krivelevich, Kwan and Sudakov~\cite{kks1} proved an analogous result where now we wish to embed a fixed spanning tree of bounded degree.
Other properties of this model (embedding a fixed subgraph, the diameter, connectivity, Ramsey properties) have been studied, for example, in~\cite{bfkm, kst}.
In~\cite{kks2} the framework was also generalised to hypergraphs  and a number of exact results concerning perfect matchings and cycles in hypergraphs and digraphs were proven.
Further, since the paper was submitted, a range of other results in the area have been obtained~\cite{bhkm, benn, bhkmpp, bmpp2, hanzhao, joos2, mm}.

Krivelevich, Kwan and Sudakov~\cite{kks1} raised the question of determining an analogue of the  B\"ottcher--Schacht--Taraz theorem~\cite{bot} in the setting of randomly perturbed dense graphs. 
In this paper we consider an important  subcase of this problem; \emph{perfect $H$-tilings}.

\subsection{Perfect tilings in graphs and random graphs} 
Given two graphs $H$  and $G$, an \emph{$H$-tiling} in $G$ 
is a collection of vertex-disjoint copies of $H$ in $G$. An
$H$-tiling is called \emph{perfect} if it covers all the vertices of $G$.
Perfect $H$-tilings are also referred to as \emph{$H$-factors} or \emph{perfect $H$-packings}. 
A seminal result in extremal graph theory is the  celebrated Hajnal--Szemer\'edi theorem~\cite{hs} which determines the minimum degree threshold that ensures a graph
contains a perfect $K_r$-tiling. Building on this result, K\"uhn and Osthus~\cite{kuhn, kuhn2}  characterised, up to an additive constant, the minimum degree which ensures that a graph $G$ 
contains a perfect $H$-tiling for an \emph{arbitrary} graph $H$.

The perfect tiling problem for random graphs has also received significant attention. 
The most striking result in the area is a theorem of Johansson, Kahn and Vu~\cite{jkv} which determines the threshold for the property that the Erd\H{o}s--R\'enyi random graph $G_{n,p}$
a.a.s.~contains a perfect $H$-tiling where $H$ is any fixed \emph{strictly balanced} graph.  
Recently, Gerke and McDowell~\cite{gm} determined the corresponding threshold in the case when $H$ is a \emph{nonvertex-balanced} graph. A number of  subcases for both of these aforementioned results had been earlier proved e.g.~in~\cite{alonyuster, er, jlr, lr}. Since it is instructive to compare both of these results to our main result, we will formally state them below.

Suppose that $H$ is a graph on at least two vertices. We write $e(H)$ and $|H|$ for the number of edges and vertices in $G$ respectively. Define
$$d(H):=\frac{e(H)}{|H|-1} \ \ \text{ and } \ \ d^*(H):=\max \{ d(H') :{H'\subseteq H}, \, |H'|\geq 2\} .$$
We say that $H$ is \emph{strictly balanced} if for every proper subgraph $H'$ of $H$ with at least two vertices, $d(H')<d(H)$.
$H$ is \emph{balanced} if $d^*(H)=d(H)$. For $v\in V(H)$, let
$$d^*(v,H):=\max \{ d(H') :{H'\subseteq H}, \, |H'|\geq 2, \, v\in V(H')\}.$$
A graph $H$ is \emph{vertex-balanced} if, for all $v \in V(H)$, $d^*(v,H)=d^*(H)$. 

Notice that if $H$ is balanced then it is vertex-balanced. Further, if $H$ is nonvertex-balanced then it is not balanced and not strictly balanced. 
The following seminal result of Johansson, Kahn and Vu~\cite{jkv} determines the threshold for the property that  $G_{n,p}$
a.a.s.~contains a perfect $H$-tiling where $H$ is an arbitrary fixed strictly balanced graph.  
\begin{thm}[Johansson, Kahn and Vu~\cite{jkv}]\label{thmjkv}
Let $H$ be a strictly balanced graph with $m$ edges and let $n\in \mathbb N$ be divisible by $|H|$.
\begin{itemize}
\item If $p \gg n^{-1/d(H)}(\log n)^{1/m}$ then a.a.s.~$G_{n,p}$ contains a perfect $H$-tiling.
\item If $p \ll n^{-1/d(H)}(\log n)^{1/m}$ then a.a.s.~$G_{n,p}$ does not contain a perfect $H$-tiling.
\end{itemize}
\end{thm}
Note here (and elsewhere where we consider bounds on $p$), we write e.g. $p \ll n^{-1/d(H)}(\log n)^{1/m}$ to mean $p=o(n^{-1/d(H)}(\log n)^{1/m})$. (Later we also use the $\ll$ notation in hierarchies of constants; this is defined in Section~\ref{note}.)

Johansson, Kahn and Vu~\cite{jkv} conjectured that Theorem~\ref{thmjkv} can be generalised to \emph{all} vertex-balanced graphs $H$:
Given any vertex $v \in V(H)$, define $$s_v:=\min \{ e(H') :{H'\subseteq H}, \, |H'|\geq 2, \, v\in V(H') \text{ and } d(H')=d^*(v,H)\}$$ and let $s$ be the maximum of the $s_v$'s amongst all $v \in V(H)$.
They conjectured that $n^{-1/d^*(H)}(\log n)^{1/s}$ is the threshold for the property that  $G_{n,p}$
a.a.s.~contains a perfect $H$-tiling where $H$ is an arbitrary fixed vertex-balanced graph. Note that Theorem~\ref{thmjkv} is a special case of this conjecture.

For nonvertex-balanced graphs, the following result determines the corresponding threshold.
 \begin{thm}[Gerke and McDowell~\cite{gm}]\label{thmgm}
Let $H$ be a nonvertex-balanced graph and let  $n\in \mathbb N$ be divisible by $|H|$.
\begin{itemize}
\item If $p \gg n^{-1/d^*(H)}$ then a.a.s.~$G_{n,p}$ contains a perfect $H$-tiling.
\item If $p \ll n^{-1/d^*(H)}$ then a.a.s.~$G_{n,p}$ does not contain a perfect $H$-tiling.
\end{itemize}
\end{thm}

\subsection{Tilings in randomly perturbed dense graphs}
The main result of this paper is to determine how many random edges one needs to add to a graph of linear minimum degree to ensure it contains a perfect $H$-tiling for \emph{any} fixed graph $H$.
Throughout the paper we assume that an $n$-vertex graph has vertex set $[n]$ and if $G$ and $G'$ are $n$-vertex graphs then we define $G\cup G'$ to be the $n$-vertex (simple) graph with edge set $E(G) \cup E(G')$.
We are now ready to state our main result.
\begin{thm}\label{mainthm}
Let $H$ be a fixed graph with at least one edge and let $n \in \mathbb N$ be divisible by $|H|$.
For every $\alpha >0$, there is a $c=c(\alpha,H)>0$ such that if $p\geq cn^{-1/d^*(H)}$ and
$G$ is an $n$-vertex graph with minimum degree $\delta (G) \geq \alpha n$
then a.a.s.~$G\cup G_{n,p}$ contains a perfect $H$-tiling.
\end{thm}
Theorem~\ref{mainthm} is best-possible in the sense that, for any fixed graph $H$, there are $n$-vertex graphs $G$ of linear minimum degree such that if $p\ll n^{-1/d^*(H)}$ then a.a.s.~$G\cup G_{n,p}$ does not contain a perfect $H$-tiling. We explain this in more detail in Section~\ref{extremeeg}. 

In the case when $H$ is strictly balanced notice that unlike Theorem~\ref{thmjkv}, Theorem~\ref{mainthm} does not involve a logarithmic term. Thus comparing  our model with the  Erd\H{o}s--R\'enyi model,
we see that starting with a graph of linear minimum degree instead of the empty graph
saves a logarithmic factor in terms of how many random edges one needs to ensure the resulting graph a.a.s.~contains a perfect $H$-tiling.
This same phenomenon is also exhibited in the analogous problems for Hamilton cycles~\cite{bfm1} and spanning trees~\cite{kks1}, as well as for matchings and loose cycles in the hypergraph setting~\cite{kks2}. Further, if the Johansson, Kahn and Vu conjecture is true then  together
with Theorem~\ref{mainthm} this shows that the same phenomenon occurs for perfect $H$-tilings for any vertex-balanced $H$. Interestingly though the threshold in Theorem~\ref{thmgm} is the \emph{same} as that
in Theorem~\ref{mainthm}. In other words, if $H$ is nonvertex-balanced, starting with a graph of linear minimum degree instead of the empty graph essentially
provides \emph{no benefit} in terms of how many random edges one needs to ensure the resulting graph a.a.s.~contains a perfect $H$-tiling!

It is also instructive to compare Theorem~\ref{mainthm} to the problem of finding an almost perfect tiling in the random graph:
We say that  $G_{n,p}$ has an \emph{almost perfect $H$-tiling} if for every $\eps>0$, the probability that the largest $H$-tiling in $G_{n,p}$ covers less than $(1-\eps)n$ vertices tends to zero as $n$ tends to infinity.
Ruci\'nski~\cite{rucinski} proved that $n^{-1/d^*(H)}$ is the threshold for $G_{n,p}$ having an almost perfect $H$-tiling (for any fixed graph $H$). Thus, in Theorem~\ref{mainthm} one can already guarantee an almost perfect $H$-tiling without using any of the
 edges from $G$. Hence (in the case when $H$ is strictly balanced), the edges in $G$ are necessary to `transform' an almost perfect $H$-tiling in $G_{n,p}$ into a perfect $H$-tiling in $G\cup G_{n,p}$.

The  proof of Theorem~\ref{mainthm} utilises Szemer\'edi's Regularity lemma~\cite{szemeredi} as well as a special case of a result of Koml\'os~\cite{komlos} concerning almost perfect $H$-tilings in dense graphs.
We also draw on ideas from~\cite{triangle, kks1}. In Section~\ref{seco} we give an overview of the proof.

\subsection{Notation}\label{note}
Let $G$ be a graph.
We write $V(G)$ for the vertex set of $G$, $E(G)$ for the edge set of $G$ and define
$|G|:=|V(G)|$ and $e(G):= |E(G)|$. Given a subset $X \subseteq V(G)$, we write $G[X]$ for the subgraph of $G$ induced by $X$. Given some $x \in V(G)$ we write 
$G-x$ for the subgraph of $G$ induced by $V(G)\setminus \{x\}$.
The degree of $x $ is denoted by $d_G(x)$ and its neighbourhood
by $N_G(x)$. Given a vertex $x \in V(G)$ and a set $Y \subseteq V(G)$ we write $d _G (x,Y)$ to denote the number of edges $xy$ where $y \in Y$. 
Given disjoint vertex classes $X,Y\subseteq V(G)$,
we write $G[X,Y]$ for the bipartite graph with vertex classes $X$ and $Y$ whose edge set consists of all those edges in $G$
with one endpoint in $X$ and the other in $Y$;
 we write $e_G(X,Y)$ for the number of edges in $G[X,Y]$.
Given a set $X$ and $t\in \mathbb N$, let $\binom{X}{t}$ denote the set of all subsets of $X$ of size $t$.

We write $0<\alpha \ll \beta \ll \gamma$ to mean that we can choose the constants
$\alpha, \beta, \gamma$ from right to left. More
precisely, there are increasing functions $f$ and $g$ such that, given
$\gamma$, whenever we choose some $\beta \leq f(\gamma)$ and $\alpha \leq g(\beta)$, all
calculations needed in our proof are valid. 
Hierarchies of other lengths are defined in the obvious way.
Throughout the paper we omit floors and ceilings whenever this does not affect the
argument.

\smallskip

The paper is organised as follows. In Section~\ref{secover} we give an overview of the proof of Theorem~\ref{mainthm}; we also give an example that shows the bound on $p$ in Theorem~\ref{mainthm} is best possible up to a  multiplicative constant.
 Szemer\'edi's Regularity lemma is presented in Section~\ref{secreg} and then  we introduce some useful tools in Section~\ref{extreme}. The proof of Theorem~\ref{mainthm} is  given in Section~\ref{secproof}.

\section{An extremal example and overview of the proof}\label{secover}
\subsection{An extremal example}\label{extremeeg}
In this subsection we prove that Theorem~\ref{mainthm} is best-possible in the sense that, given any graph $H$, there exist (sequences of) $n$-vertex graphs $G_n$ with linear minimum degree so that, if $p\ll n^{-1/d^*(H)}$, then a.a.s.~$G_n\cup G_{n,p}$ does not contain a perfect $H$-tiling.
For this we require the following result. Recall the definition
$$d^*(H):=\max\bigg\{\frac{e(H')}{v(H')-1}: H'\subseteq H, |H'|\geq 2\bigg\}.$$
\begin{thm}[\label{jansonthm3}\cite{jlr}, part of Theorem $4.9$]
For every graph $H$ with at least one edge and for every $0<\eps<1$ there is a positive constant $c=c(H,\eps)$ such that if $p\leq cn^{-1/d^*(H)}$,
$$\lim_{n\rightarrow \infty} \mathbb{P}(G_{n,p} \ \text{contains an $H$-tiling covering at least $\eps n$ vertices})=0.$$
\end{thm}
Consider any 
fixed graph   $H$ with at least one edge and let $b>a>0$ such that $a+b=1$ and $b>a(|H|-1)$. Set $\eps:=b-a(|H|-1)>0$ 
and define $c=c(H,\eps)$ as in Theorem~\ref{jansonthm3}. Let $n\in \mathbb N$ be divisible by $|H|$ and 
let $G_n$ be the complete  bipartite graph with vertex classes $X$ and $Y$ of sizes $an$ and $bn$ respectively.
Consider $G'_n:=G_n\cup G_{n,p}$ where $p\leq c(bn)^{-1/d^*(H)}$. Notice that if $G'_n$ contains a perfect $H$-tiling, then $G_{n,p}[Y]\cong G_{bn,p}$ must contain an $H$-tiling covering at least $\eps n> \eps (bn)$ vertices.
However, by the choice of $c$, Theorem~\ref{jansonthm3} implies that a.a.s.~such an $H$-tiling in $G_{n,p}[Y]$ does not exist. Thus, a.a.s.~$G'_n$ does not contain a perfect $H$-tiling.
So we have indeed shown that the bound on $p$ in Theorem~\ref{mainthm} is best-possible.

The above example shows that if $0<\alpha < 1/|H|$ then we need the random edges with $p\geq Cn^{-1/d^*(H)}$  to force a perfect $H$-tiling. We also note that if $\alpha > 1-\frac{1}{\chi(H)}$ and $n$ is large then by a theorem of Koml\'os,
S\'ark\"ozy and Szemer\'edi~\cite{komlossarkozy}, every $n$-vertex graph $G_n$ of minimum degree $\alpha n$ contains a perfect $H$-tiling and thus there is no need for the random edges at all.  
More generally, given any graph $H$, K\"uhn and Osthus~\cite{kuhn, kuhn2} determined the smallest
$\alpha ^*=\alpha^*(H)>0$  such that, given any $\alpha >\alpha^*$,   every sufficiently large $n$-vertex graph $G_n$ of minimum degree at least $\alpha n$ contains a perfect $H$-tiling.
It may be of interest to investigate the following question: given fixed $\alpha$ with $1/|H| < \alpha < \alpha^*$, how large must $p$ be to ensure that whenever $G_n$ has minimum degree at least $\alpha n$ then $G_n\cup G_{n,p}$ a.a.s.~contains a perfect $H$-tiling? 

Another natural question is whether Theorem~\ref{mainthm} holds if we replace $\alpha n$ with a sublinear term. Note that the approach we use is not suitable for attacking this problem (since we apply Szemer\'edi's Regularity lemma). 

\subsection{Overview of the proof of Theorem~\ref{mainthm}}\label{seco}
The first step of the proof is to obtain some special structure within $G$ to help us embed the perfect $H$-tiling. In particular, by applying Szemer\'edi's Regularity lemma (Lemma~\ref{regu}) and a theorem of Koml\'os (Theorem~\ref{komlos})
we obtain a spanning subgraph of $G$ which `looks' like the blow-up of a collection of stars. 
More precisely, there is a spanning subgraph $G'$ of $G$; constants
 $k,t \in \mathbb N$, and; a partition of $V(G)$ into classes $V_0$ and $W_{i,j}$
 (for all 
$0 \leq i \leq t$ and $1\leq j \leq k$), 
 such that: 
\begin{itemize}
\item $V_0$ is `small';
\item Each cluster $W_{i,j}$ has the same size;
\item For all $1\leq i \leq t$, $1\leq j \leq k$, the pair $(W_{0,j},W_{i,j})_{G'}$ is `super-regular'.
\end{itemize}
We remark that a similar structure was used  in~\cite{kks1} (though the 
role of the blown-up stars was different there).

If $t=1$, then the aforementioned structure would consist of a collection of disjoint
super-regular pairs $(W_{0,j},W_{1,j})_{G'}$ and an `exceptional set' $V_0$.
We could then obtain a perfect $H$-tiling using the strategy described below.

The first step is to find a small $H$-tiling
$\mathcal H$ that covers all of $V_0$ but so that $\mathcal H$ only intersects the super-regular pairs $(W_{0,j},W_{1,j})_{G'}$ in a very small number of vertices.
To see that such an $H$-tiling $\mathcal H$ in $G \cup G_{n,p}$ exists, note that given any $v \in V_0$, $N_{G}(v)$ has linear size so a.a.s, $G_{n,p}$ contains many copies of $H-x$ in $N_{G}(v)$ (for some $x \in V(H)$). In particular, this implies that $v$ lies in many copies of $H$ in $G \cup G_{n,p}$. Thus, this property allows us to  greedily construct $\mathcal H$ (though some care is needed to ensure we do not use too many vertices in any one cluster $W_{i,j}$).

Now if we remove all the vertices lying in $\mathcal H$ we still have that each
$(W_{0,j},W_{1,j})_{G'}$ is super-regular. This structure can then be used to find
an $H$-tiling $\mathcal H_j$ in $G\cup G_{n,p}$ covering precisely the vertices of $W_{0,j}\cup W_{1,j}$. Indeed, in this case we employ an approach very similar to that used in~\cite{triangle}. 
Then $\mathcal H$, $\mathcal H_1$, \dots, $\mathcal H_k$ together
form a perfect $H$-tiling in $G\cup G_{n,p}$, as desired.

In particular, note that roughly speaking the authors of~\cite{triangle} prove that if $(A,B)$ is a (very dense) super-regular pair in a graph $G$ of small independence number, then there is a triangle-tiling in $G$ covering precisely the vertices of $A\cup B$.
Here the small independence number ensures we have large matchings in both $G[A]$ and $G[B]$; then the edges between $A$ and $B$ can be used to extend such edges to triangles with at least one vertex in each class, and thus ultimately (with significant care) one obtains the desired triangle-tiling.
In our setting, the edges from $G_{n,p}$ ensure that, given any super-regular pair $(A,B)$ in $G$, we have large  $(H-x)$-tilings in both $G_{n,p}[A]$ and $G_{n,p}[B]$ (for some $x \in V(H)$). We then (again with some care) extend such copies of $H-x$ to copies of $H$ using the edges between $A$ and $B$ in $G$, to obtain the desired $H$-tiling.

To employ the approach used in~\cite{triangle}, we really do require that
 $t=1$. That is, the structure in $G$ looks like a blow-up of disjoint edges. However, since the minimum degree of $G$ is typically very small we can only ensure a structure in $G$ that looks like a blow-up of stars, each of which contains a huge constant number of leaves (i.e.~$t$ is large). 

Instead, we have to first carefully choose a large $H$-tiling $\mathcal H'$ in $G\cup G_{n,p}$
so that what remains uncovered by $\mathcal H'$ is precisely a collection of clusters that, in $G'$, form
disjoint super-regular pairs. We then can proceed as described above.
To construct a suitable $H$-tiling $\mathcal H'$, our proof heavily uses the blown-up star structure we found initially in $G$.

\section{The Regularity lemma}\label{secreg}
In the proof of our main result we will make use of the Szemer\'{e}di's Regularity lemma~\cite{szemeredi}, hence in this section we introduce the necessary notation and set-up for this lemma. 
The \emph{density} of a bipartite graph $G$ with vertex classes $A$ and $B$ is defined as 
$$d_G(A,B):=\frac{e(A,B)}{|A||B|}.$$ 
Given any $\eps>0$ we say that $G$ is $\eps$\emph{-regular} if for all sets $X\subseteq A$ and $Y\subseteq B$ with $|X|\geq \eps|A|$ and $|Y|\geq \eps |B|$ we have $|d_G(A,B)-d_G(X,Y)|<\eps$. In this case we also say that $(A,B)_G$ is an $\eps$\emph{-regular pair}. Given $d\in [0,1)$ we say that $G$ is $(\eps,d)$\emph{-super-regular} if all sets $X\subseteq A$ and $Y\subseteq B$ with $|X|\geq \eps |A|$ and $|Y|\geq \eps |B|$ satisfy $d_G(X,Y)>d$ and moreover $d_G(a)>d|B|$ and $d_G(b)>d|A|$ for all $a\in A$ and $b\in B$. 

The following degree form of the Regularity lemma can be easily derived from the classical version and will be particularly useful for us.

\begin{lemma}[\label{regu}Regularity lemma]
For every $\eps>0$ and each integer $\ell_0$ there is an $M = M(\eps, \ell_0)$ such that if $G$ is any graph on at least $M$ vertices and $d \in [0, 1)$, then there exists a partition of $V(G)$ into $\ell+1$ classes $V_0, V_1, \ldots, V_\ell$, and a spanning subgraph $G' \subseteq G$ with the following properties:
\begin{itemize}
\item[(i)] $\ell_0 \leq \ell \leq M$, $|V_0|\leq \eps |G|$, $|V_1|=\ldots =|V_\ell|=:L$;
\item[(ii)] $d_{G'}(v)>d_G(v)-(d+\eps)|G|$ for all $v\in V(G)$;
\item[(iii)] $e(G'[V_i])=0$ for all $i\geq 1$;
\item[(iv)] for all $1\leq i < j \leq \ell$ the graph $(V_i,V_j)_{G'}$ is $\eps$-regular and has density either $0$ or greater than $d$.
\end{itemize}
\end{lemma}
The sets $V_1,\ldots, V_\ell$ are called \emph{clusters}, $V_0$ is called the \emph{exceptional set} and the vertices in $V_0$ \emph{exceptional vertices}. We refer to $G'$ as the \emph{pure graph} of $G$. 
The \emph{reduced graph} $R$ of $G$ is the graph whose vertices are $V_1, \ldots , V_\ell$ and in which $V_i$ is adjacent to $V_j$ whenever $(V_i, V_j )_{G'}$ is $\eps$-regular and has density greater than $d$.

Next we see that given a regular pair we can approximate it by a super-regular pair. The following lemma can be found in e.g.~\cite{andrewregu2}.
\begin{lemma}\label{superregular}
If $(A,B)$ is an $\eps$-regular pair with density $d$ in a graph $G$ (where $0<\eps<1/3$), then there exists $A'\subseteq A$ and $B'\subseteq B$ with $|A'|\geq (1-\eps)|A|$ and $|B'|\geq (1-\eps) |B|$, such that $(A',B')$ is a $(2\eps, d-3\eps)$-super-regular pair.
\end{lemma}

\section{Some useful results}\label{extreme}
\subsection{Almost perfect star tilings}
An important result of Koml\'os~\cite{komlos} determines the minimum degree threshold
that forces an \emph{almost} perfect $H$-tiling in a graph, for any fixed graph $H$.
As hinted at in the proof overview, we require that the reduced graph $R$ of $G$ contains an almost perfect tiling of stars. The following special case of Koml\'os' theorem ensures such a tiling exists.

\begin{thm}[\label{komlos}Koml\'{o}s~\cite{komlos}]
Given any $t\in \mathbb N$ and $\eps>0$, there is an integer $n_0=n_0(t,\eps)$ so that, if $n\geq n_0$ and  $G$ is a graph on $n$ vertices with
 $$\delta(G)\geq\frac{n}{(t+1)},$$ 
then $G$ contains a $K_{1,t}$-tiling that covers all but at most $\eps n$ vertices.
\end{thm}

\subsection{Embedding $H$ in random graphs}
Let $\eta>0$. Given an $n$-vertex graph $G$ and a graph $H$ we  write  $G\in F_H(\eta)$ if every induced subgraph of $G$ on at least $\eta n$ vertices contains $H$ as a (not necessarily induced) subgraph. 
The next theorem will allow us to find almost perfect $H$-tilings in large subgraphs of $G_{n,p}$.
\begin{thm}[\label{jansonthm1}\cite{jlr}, part of Theorem $4.9$]
For every graph $H$ with at least one edge and for every $\eta>0$ there is a positive constant $C=C(H,\eta)$ such that if $p\geq Cn^{-1/d^*(H)}$,
$$\lim_{n\rightarrow \infty} \mathbb{P}(G_{n,p}\in F_H(\eta))=1.$$
\end{thm}

Let $\gamma>0$ and $H$ be a graph.
Consider a graph $G$ with vertex set $[n]$. Let $\mathcal H_n$ be a collection of  copies of $H$ in $K_n$ (where we view $K_n$ to have vertex set $[n]$).
We  write  $G\in F_H(\gamma,  \mathcal H_n)$ if every induced subgraph of $G$ on at least $\gamma n$ vertices contains a copy of $H$ that is not an element of $\mathcal H_n$. 
The next theorem can be proven in the same way as Theorem~\ref{jansonthm1} so we omit the proof. It allows us to find copies of $H$ in $G_{n,p}$ that avoid certain edge sets.
\begin{thm}\label{jansonthm4}
Let $H$ be a graph with at least one edge and suppose $0<\gamma _1 \ll \gamma _2 \ll 1/|H|$. Then there is a positive constant 
$D=D(H,\gamma _1,\gamma _2)$ such that the following holds.
Let $\mathcal H_n$ be a collection of at most $\gamma _1 \binom{n}{|H|}$ copies of $H$ in $K_n$. If $p\geq Dn^{-1/d^*(H)}$, then
$$\lim_{n\rightarrow \infty} \mathbb{P}(G_{n,p}\in  F_H(\gamma_2, \mathcal H_n))=1.$$
\end{thm}
We remark that the proof of Theorem~\ref{jansonthm1} is just a simple  application of Theorem 3.9 from~\cite{jlr}. To prove Theorem~\ref{jansonthm4} one can follow precisely the same proof, however, by instead applying a version of Theorem 3.9 from~\cite{jlr}
in the setting where now some copies of $H$ are excluded; again to prove such a result one follows the proof of Theorem 3.9 from~\cite{jlr} precisely.

Similarly to the above, given an $n$-vertex graph $G$ and a graph $H$ we write  $G\in F'_H(\eta)$ if for all ordered pairs of disjoint sets $A,B\subset V(G)$ of size $|A|,|B|\geq \eta n$ there exists a copy of $H$ in $G$ with precisely one vertex in $A$ and $|H|-1$ vertices in $B$. Again, the same argument as in the proof of Theorem~\ref{jansonthm1} shows the following.
\begin{thm}\label{jansonthm2}
For every graph $H$ with at least one edge and for every $\eta>0$ there is a positive constant $C=C(H,\eta)$ such that if $p\geq Cn^{-1/d^*(H)}$,
$$\lim_{n\rightarrow \infty} \mathbb{P}(G_{n,p}\in F'_H(\eta))=1.$$
\end{thm}
\COMMENT{The proof of Theorem~\ref{jansonthm1} is just a simple application of Theorem 3.9 from the Random graphs book~\cite{jlr}. So it suffices to prove an
analogous version of Theorem 3.9 but e.g.~in the $|H|$-partite setting (i.e.~asking for the probability that the random $|H|$-partite graph $G_{n,|H|,p}$ with n vertices in each class has no copy of $H$ in it). Such a version of Thm 3.9 follows in the same way as the original proof.}

\section{Proof of Theorem~\ref{mainthm}}\label{secproof}
Let $H$ and $\alpha >0$ be as in the statement of the theorem.
Note that it suffices to prove the theorem in the case when $\alpha \ll 1/|H|$.
Define additional constants $\phi,\ell_0,\eps,  \eps _1,\eps_2, \eps _3,\eps_4,\eps_5, d_1, d$ and apply the Regularity lemma (Lemma~\ref{regu}) with input $\eps, \ell _0$ to
obtain  $M=M(\eps,\ell_0)$ and define $\eta>0$ such that $\eta \ll 1/M$, so that we have

\begin{align}\label{hier}
0<\eta  \ll 1/M  \leq 1/\ell_0 \ll\eps \ll \eps_1 \ll \eps_2 \ll \eps_3 \ll \eps_4 \ll \phi \ll \eps_5  \ll d_1 \ll d\ll \alpha\ll 1/|H|.
\end{align}
Set $$t := \lceil 2\alpha^{-1}\rceil$$
and let $H'$ be some fixed induced subgraph of $H$ on $|H|-1$ vertices. 
Let $c=c(\eta, \eps _2, \eps_3, M,t,H)=c(\alpha,H)$ be a positive constant such that: (i) on input $H$, $\eta$,
the conclusion of both Theorem~\ref{jansonthm1} and Theorem~\ref{jansonthm2} hold with $c/4$ playing the
role of $C$ and; (ii) $c\geq 4D(4Mt)^{1/d^*(H)}$ where $D$ is the output of Theorem~\ref{jansonthm4}  
on input $H'$, $\eps_2$ and $\eps_3$.\COMMENT{condition really needed}

Let $G$ be a sufficiently large $n$-vertex graph with $n$ divisible by $|H|$ and
$\delta (G) \geq \alpha n$. 
Set
$$p:=cn^{-1/d^*(H)}.$$
 We wish to show that $G\cup G_{n,p}$ a.a.s.~contains a perfect $H$-tiling. Our first step towards proving this will be to switch our attention to an appropriate subgraph of $G$.

\begin{claim}\label{claim1}
There is a $k \in\mathbb N$,  a partition of $V(G)$ into classes $V_0$ and $W_{i,j}$ (for all 
$0 \leq i \leq t$ and $1\leq j \leq k$), and a spanning subgraph $G'$ of $G$
 such that: 
\begin{itemize}
\item[(i)] $\frac{\ell_0}{2t}  \leq  k 
\leq  \frac{M}{t+1}$;
\item[(ii)] $|V_0|\leq 2t\eps n$;
\item[(iii)] Each cluster $W_{i,j}$ has the same size $L \in \mathbb N$
where $n/2M\leq L \leq n/\ell_0$;
\item[(iv)] For all $1\leq i \leq t$, $1\leq j \leq k$, the pair $(W_{0,j},W_{i,j})_{G'}$ is $(4\eps,d/2)$-super-regular;
\item[(v)] $\delta (G') \geq (\alpha -2d)n$.
\end{itemize}
\end{claim}
\begin{proof}
Apply the Regularity lemma (Lemma~\ref{regu}) to $G$ with parameters $\eps, d$ and $\ell_0$ to obtain a partition $V_0,V_1,\ldots,V_\ell$ of $V(G)$, pure graph $G'$ of $G$ and the reduced graph $R$ of $G$. 
So $|V_0|\leq \eps n$ and $(1-\eps)n/\ell \leq L':=|V_i|=|V_j|\leq n/\ell$ for all $i ,j\geq 1$.
It is a well-known fact that the reduced graph $R$ of $G$ `almost' inherits the minimum degree of $G$ (see e.g.~\cite[Lemma~3.7]{andrewregu2}). In particular,
since $\delta(G)\geq \alpha n$ and $\eps \ll d \ll \alpha$ we have that 
$$\delta(R)\geq  \frac{\alpha\ell}{2}\geq \frac{\ell}{t+1} .$$
By (\ref{hier}), $\ell \geq \ell _0$ is sufficiently large compared to $1/\eps$ and
$t$.\COMMENT{AT: important... needed to apply Komlos' theorem}
Thus,
 Theorem~\ref{komlos} implies that $R$ contains a $K_{1,t}$-tiling $\mathcal K$ that covers all but at most $\eps \ell$ vertices of $R$.
Let $k$ denote the size of $\mathcal K$. Hence,
$$\frac{\ell_0}{2t}  \leq \frac{(1-\eps)\ell}{t+1} \leq k \leq \frac{\ell}{t+1}
\leq  \frac{M}{t+1}.$$
Move to $V_0$ all those vertices that lie in clusters uncovered by $\mathcal K$.
Hence
$$|V_0|\leq \eps n+ \eps \ell \times L' 
\leq \eps n+ \eps \ell \times \frac{n}{\ell} =2\eps n.$$

Consider a copy $K'_{1,t}$ of $K_{1,t}$ in $\mathcal K$. In $G'$, this copy of $K_{1,t}$ corresponds to a collection of $t+1$ clusters $V_{i_0},V_{i_1},\ldots,V_{i_t}$ such that for all $1\leq j \leq t$ the pair $(V_{i_0},V_{i_j})_{G'}$ is $\eps$-regular and has density greater than $d$. 
By repeatedly applying Lemma~\ref{superregular} it is easy to check that we can
remove precisely $\eps t|V_{i_j}|=\eps t L'$ vertices from $V_{i_j}$ (for each $0 \leq j \leq 
t$) so that now $(V_{i_0},V_{i_j})_{G'}$ is $(4\eps,d/2)$-super-regular 
for each $1\leq j \leq t$. Add the vertices removed from these clusters into $V_0$.
Repeat this process for all copies of $K_{1,t}$ in $\mathcal K$.
Thus, now
$$|V_0|\leq 2\eps n+ \eps t L' \ell \leq 2 t\eps  n. $$
Given the $j$th copy of $K_{1,t}$ in $\mathcal K$ we relabel the clusters
so that the root of this $K_{1,t}$ is $W_{0,j}$ and the leaves are $W_{1,j},\dots
, W_{t,j}$. 
Hence  (iv) holds. Further, each of the clusters $W_{i,j}$ has the same size $L\in \mathbb N$ where
$$\frac{n}{2M} \leq \frac{(1-\eps)(1-\eps t)n}{\ell} \leq L =(1-\eps t)L'
\leq \frac{n}{\ell}
 \leq \frac{n}{\ell_0},$$
so (iii) holds. Note that (v) holds by Lemma~\ref{regu}(ii) and as $\eps \ll d$.
\end{proof}

So far we have only used the deterministic edges, i.e.~edges in $G$, but recall that we are aiming to find a perfect $H$-tiling in $G\cup G_{n,p}$. Next we will use these random edges to find copies of $H':=H-x$ (for some $x\in V(H)$) in the neighbourhood of vertices in $V_0$. To simplify the later calculations we now use the standard trick of decomposing the random edges into a few `buckets'. That is, let $G_1,G_2,G_3,G_4$ be independently chosen elements of $G_{n,p/4}$ and observe that $G_1\cup G_2\cup G_3\cup G_4$ has the same distribution as $G_{n,p'}$ for some $p'\leq p$. Hence it suffices to consider the graph $G\cup (G_1\cup G_2\cup G_3\cup G_4)$ instead of the graph $G\cup G_{n,p}$.
Define $\mathcal W:=\{W_{i,j}:  
0 \leq i \leq t\text{ and } 1\leq j \leq k\}$
and $k':=|\mathcal W|=k(t+1)$.
\begin{claim}\label{claim2}
Asymptotically almost surely there is a set $Z\subseteq V(G)$ such that: 
\begin{itemize}
\item[(i)] $(G'\cup G_1)[Z]$ contains a perfect $H$-tiling $\mathcal H_1$; 
\item[(ii)] $V_0\subseteq Z$, and for all $W_{i,j} \in \mathcal W$ we have that $|Z\cap W_{i,j}|\leq \eps_1 L$.
\end{itemize}
\end{claim}
\begin{proof}
Define an auxiliary bipartite graph $Q$ with vertex classes $V_0$ and $\mathcal W$ in which a vertex $v\in V_0$ is adjacent to a cluster $W_{i,j}$ precisely if $d_{G'}(v,W_{i,j})\geq \alpha L/4$.
For each $v \in V_0$, $d_{G'}(v)\geq (\alpha-2d)n$ and is adjacent to at most $2t\eps n$ vertices in $V_0$. Therefore,
$$3\alpha n/4 \stackrel{(\ref{hier})}{\leq }d_{G'}(v)-2t\eps n \leq  L d_Q(v)+k'\alpha L/4 \leq L d_Q(v)+\alpha n/4, $$
and so 
$d_Q(v)\geq \alpha k'/2$.

Hence we can find an assignment $f: V_0\rightarrow \mathcal W$ such that $vf(v)$ is an edge in $Q$ for all $v \in V_0$ and for any $W_{i,j} \in \mathcal W$, 
\begin{equation}\label{inversesmall}
|f^{-1}(W_{i,j})|\leq \frac{4t\eps n}{\alpha k'} \stackrel{(\ref{hier})}{\leq } \frac{\eps _1 L}{|H|}.
\end{equation}
Here in the first inequality we use that $|V_0|\leq 2t\eps n$ and  $d_Q(v)\geq \alpha k'/2$ for all $v\in V_0$; in the last inequality we use that $k'L=n-|V_0|\geq (1-2t\eps)n$.

Enumerate the vertices $v_1,\dots, v_s$ of $V_0$. For each such $v_i$, we will obtain a copy $H_i$ of $H$ in $G\cup G_1$ so that:
\begin{itemize}
\item $H_i$ contains $v_i$ for each $1\leq i \leq s$, and all the other vertices in $H_i$ lie in the cluster $f(v_i) \in \mathcal W$;
\item $H_i$ and $H_j$ are vertex-disjoint for all $1\leq i\not = j \leq s$.
\end{itemize}
Note that finding such copies of $H$ would immediately prove the claim. Indeed, we then define $Z$ to consist of all the vertices in the $H_i$. In particular, (\ref{inversesmall}) then implies that $|Z\cap W_{i,j}|\leq \frac{\eps_1 L}{|H|} \times (|H|-1)\leq \eps _1 L$ for
each $W_{i,j} \in \mathcal W$.

Suppose for some $1\leq j <s$ we have constructed $H_1,\dots, H_{s-1}$ as desired. Consider $v_s\in V_0$ and let $W_{i_s,j_s}:=f(v_s)$. Then by the definition of $Q$ and $f$, there is a set $W \subseteq W_{i_s,j_s}$ so that
$|W|\geq \alpha L/4-\eps _1 L\geq \alpha L/5$; $W$ is disjoint from $H_1,\dots, H_{s-1}$ and; $W\subseteq N_{G'}(v_s)$.

Note that $H':=H-x$ (for some $x \in V(H)$)  either  consists of isolated vertices or $0<d^*(H')\leq d^*(H)$. Further $\eta n  \leq \alpha n/(10M)\leq  \alpha L/5$ by (\ref{hier}) and Claim~\ref{claim1}(iii). 
So Theorem~\ref{jansonthm1} implies that a.a.s, $G_1[W]$ contains a copy of $H'$. Since in $G'$, $v_s$ is adjacent to every vertex in $W$, this yields the desired copy $H_s$ of $H$ in $G\cup G_1$. 
\end{proof}

For all $i,j$ set 
$$V_{i,j}:=W_{i,j}\setminus Z.$$ 
Since $(G'\cup G_1)[Z]$ contains a perfect $H$-tiling $\mathcal H_1$, it suffices to prove that, a.a.s, there is an $H$-tiling in $G'\cup G_2\cup G_3 \cup G_4$ that covers precisely the vertices in $V(G)\setminus Z$.
Note that by Claim~\ref{claim1}(iv) and Claim~\ref{claim2}(ii) we have that $(V_{0,j},V_{i,j})_{G'}$ is $(5\eps,d/3)$-super-regular for all $1\leq i \leq t$ and $1\leq j \leq k$.
Further,
\begin{align}\label{eq11}
(1-\eps _1)L\leq |V_{i,j}|\leq L
\end{align}
for all $0\leq i \leq t$ and $1\leq j \leq k$.

For each $1\leq j \leq k$, randomly partition $V_{0,j}$ into $t$ vertex classes $S_{1,j},\dots, S_{t,j}$ such that
\begin{align}\label{eq12}
|S_{i,j}|=\frac{|V_{0,j}|}{t}
\end{align}
\COMMENT{here ignoring floors and ceilings}
for all $1\leq i\leq t$.
For every $1\leq i \leq t$ and $1\leq j \leq k$, randomly partition $V_{i,j}$ into two vertex classes $V'_{i,j}$ and $T'_{i,j}$ so that
\begin{align}\label{eq13}
|T'_{i,j}|=\frac{|V_{i,j}|}{t}.
\end{align}

\smallskip

Roughly speaking, the rest of the proof now proceeds as follows:
For each of the clusters $V'_{i,j}$ we will find an almost perfect $H$-tiling in $G_2[V'_{i,j}]$. 
This will ensure that  almost all of the uncovered vertices in $V(G)$ lie in the clusters $S_{i,j}$ and $T'_{i,j}$ (for all $i,j$). However, as the next claim shows, each $(S_{i,j},T'_{i,j})_{G'}$ is `super-regular'. 
Then by modifying each of the super-regular pairs slightly,
this structure will allow
us to cover all remaining vertices in $V(G)$ with an $H$-tiling, thereby completing our perfect $H$-tiling.

\begin{claim}\label{claim3}
The following conditions hold a.a.s:
\begin{itemize}
\item[(i)] Let $1\leq i \leq t$ and $1\leq j \leq k$. For all sets $X\subseteq S_{i,j}$ and $Y \subseteq T'_{i,j}$ with $|X|\geq \eps _1 |S_{i,j}|$ and $|Y|\geq \eps _1 |T'_{i,j}|$, we have that
$d_{G'}(X,Y)\geq d/3$;
\item[(ii)] Let $1\leq j \leq k$. For every vertex $v \in V_{0,j}$, $d_{G'}(v,T'_{i,j})> d|T'_{i,j}|/4$ for all $1\leq i \leq t$;
\item[(iii)] Let $1\leq i \leq t$ and $1\leq j \leq k$.  For every $v \in V_{i,j}$, $d_{G'}(v,S_{i,j})> d|S_{i,j}|/4$.
\end{itemize}
In particular, (i)--(iii) imply that $(S_{i,j},T'_{i,j})_{G'}$ is $(\eps_1,d/4)$-super-regular for all $1\leq i \leq t$ and $1\leq j \leq k$.
\end{claim}
\begin{proof}
For (i), note that $X\subseteq V_{0,j}$ and $Y\subseteq V_{i,j}$ where $|X|\geq \eps _1 |S_{i,j}| \stackrel{(\ref{eq12})}{=}\eps _1 |V_{0,j}|/t\stackrel{(\ref{hier})}{\geq} 5\eps |V_{0,j}|$
and $|Y|\geq 5\eps |V_{i,j}|$. So as $(V_{0,j},V_{i,j})_{G'}$ is $(5\eps,d/3)$-super-regular, (i) follows immediately.

For (ii) note that such a $v \in V_{0,j}$ satisfies $d_{G'}(v,V_{i,j})\geq d|V_{i,j}|/3$ for all $1\leq i \leq t$. Thus by applying a Chernoff bound for the hypergeometric distribution (e.g.~\cite[Theorem 2.10]{jlr}), a.a.s.~(ii) holds. 
(iii) holds similarly.
\end{proof}
 We will now obtain an $H$-tiling $\mathcal H_2$ in $G_2$ that covers almost all of the vertices in the classes $V'_{i,j}$.
Fix $1\leq i \leq t$ and $1\leq j \leq k$. 
By repeatedly applying Theorem~\ref{jansonthm1} to $G_2$, we obtain an
$H$-tiling $\mathcal H_{i,j}$ in $G_2[V'_{i,j}]$ that covers all but at most $\eta n$ of the vertices in $V'_{i,j}$.
Let $\mathcal H_2$ denote the $H$-tiling in $G_2$ obtained by taking the disjoint union of all of the $\mathcal H_{i,j}$, and write $Z':=V(\mathcal H_2)$.

For each $1\leq i \leq t$ and $1\leq j \leq k$, let $T_{i,j}$ be obtained from $T'_{i,j}$ by adding to it all the vertices in $V'_{i,j}$ that are uncovered by $\mathcal H_{i,j}$.
Thus, $|T_{i,j}|\leq |T'_{i,j}|+\eta n$. Note that 
$|T'_{i,j}|\stackrel{(\ref{eq13})}{=}|V_{i,j}|/t \stackrel{(\ref{eq11})}{\geq} (1-\eps_1)L/t$.
On the other hand, $\eta n \leq \eps (1-\eps_1)L/t$ by (\ref{hier}) and Claim~\ref{claim1}(iii). 
Hence, $|T_{i,j}|\leq (1+\eps)|T'_{i,j}|\stackrel{(\ref{eq11}),(\ref{eq13})}{\leq} (1+\eps)L/t$. Together with Claim~\ref{claim3}(i)--(iii) this implies that 
$(S_{i,j},T_{i,j})_{G'}$ is $(2\eps_1,d/5)$-super-regular.
To summarise, so far we have proved the following.
\begin{claim}\label{claim4}
Asymptotically almost surely, there is a partition of $V(G)$ into classes $Z,$ $Z'$ and $S_{i,j}$, $T_{i,j}$ for each $1\leq i \leq t$ and $1\leq j \leq k$, such that:
\begin{itemize}
\item[(i)] There is a perfect $H$-tiling $\mathcal H_1$ in $(G'\cup G_1)[Z]$ and a perfect $H$-tiling $\mathcal H_2$ in $G_2[Z']$;
\item[(ii)] Each $(S_{i,j},T_{i,j})_{G'}$ is $(2\eps_1,d/5)$-super-regular;
\item[(iii)] $(1-\eps _1)L/t\leq |S_{i,j}|,|T_{i,j}|\leq  (1+\eps)L/t$. \qed
\end{itemize}
\end{claim}

To finish the proof of Theorem~\ref{mainthm} we wish to prove that a.a.s.~there is a perfect $H$-tiling in each of the graphs $(G'\cup G_4)[S_{i,j}\cup T_{i,j}]$.
However, for this we need that $|H|$ divides $|S_{i,j}\cup T_{i,j}|$. So we first modify the clusters $S_{i,j}, T_{i,j}$ slightly to ensure this.

Suppose that $1\leq i \leq t$ and $1\leq j \leq k$ are such that $|S_{i,j}\cup T_{i,j}|$ is not divisible by $|H|$. Then, since $|H|$ divides $n$, there are  $1\leq i' \leq t$ and $1\leq j' \leq k$ with
$(i,j)\not =(i',j')$ so that $|S_{i',j'}\cup T_{i',j'}|$ is also not divisible by $|H|$.
Recall $|S_{i,j}|, |S_{i',j'}|\gg \eta n$. Thus,
 by repeated applications of Theorem~\ref{jansonthm2}, a.a.s.~we can find a collection $\mathcal S_{i,j}$ of at most $|H|-1$ disjoint copies of $H$ in $G_3$ so that: each such $H$ has one vertex in $S_{i,j}$ and $|H|-1$ vertices in 
$S_{i',j'}$ and; after removal of the vertices in $\mathcal S_{i,j}$ we now have that $|S_{i,j}\cup T_{i,j}|$ is divisible by $|H|$. 
Continuing in this way, we  obtain an $H$-tiling $\mathcal H_3$ in $G_3$ so that:
$\mathcal H_3$ only covers vertices in the clusters $ S_{i,j}$; $|\mathcal H_3|\leq |H|kt$ and; after removal of all those vertices in $\mathcal H_3$ from each of the $ S_{i,j}$ we have that
$|S_{i,j}\cup T_{i,j}|$ is divisible by $|H|$ for all $1\leq i \leq t$ and $1\leq j \leq k$.

In particular, we now have the following updated version of Claim~\ref{claim4}.
\begin{claim}\label{claim5}
Asymptotically almost surely, there is a partition of $V(G)$ into classes $Z,$ $Z'$, $Z''$ and $S_{i,j}$, $T_{i,j}$ for each $1\leq i \leq t$ and $1\leq j \leq k$, such that:
\begin{itemize}
\item[(i)] There is a perfect $H$-tiling $\mathcal H_1$ in $(G'\cup G_1)[Z]$; a perfect $H$-tiling $\mathcal H_2$ in $G_2[Z']$ and; a perfect $H$-tiling $\mathcal H_3$ in $G_3[Z'']$;
\item[(ii)] Each $(S_{i,j},T_{i,j})_{G'}$ is $(3\eps_1,d/6)$-super-regular;
\item[(iii)] $(1-2\eps _1)L/t\leq |S_{i,j}|,|T_{i,j}|\leq  (1+\eps)L/t$;
\item[(iv)] $|S_{i,j}\cup T_{i,j}|$ is divisible by $|H|$ for all $1\leq i \leq t$ and $1\leq j \leq k$. \qed
\end{itemize}
\end{claim}

Now all that remains to prove is that there is a perfect $H$-tiling in each of the graphs $(G'\cup G_4)[S_{i,j}\cup T_{i,j}]$.
 Let $(S,T)_{G'}$ be one of our super-regular pairs and recall that $H'$ is an induced subgraph of $H$ on $h:=|H|-1$ vertices.

Call an $h$-set $\{v_1,\ldots,v_h\}\in \binom{T}{h}$ \emph{good} if $|S\cap N_{G'}(v_1)\cap\ldots \cap N_{G'}(v_h)|\geq d_1 |S|$,
otherwise we say $\{v_1,\ldots,v_h\}\in \binom{T}{h}$ is \emph{bad}. 
\begin{claim}\label{newc}
At least $(1-4\eps_1)^{h-1}\binom{|T|}{h}$ $h$-sets in $\binom{T}{h}$ are good. 
\end{claim}
\begin{proof}
Pick any vertex $v_1\in T$ and let $N_1:=N_{G'}(v_1)\cap S$. Since $(S,T)_{G'}$ is $(3\eps _1,d/6)$-super-regular, we have that $|N_1|\geq d|S|/6$ and every set $Y\subseteq T$ of size at least $3\eps_1 |T|$ contains at least one vertex $v_2$ with 
$|N_{G'}(v_2)\cap N_1|\geq d|N_1|/6$. 
So in particular at least $(1-4\eps_1)|T|$  vertices $v_2$ in $T$ are such that $|N_2|\geq d|N_1|/6\geq d^2|S|/36 $ where $N_2:=N_{G'}(v_2)\cap N_1$. 
Pick such a vertex $v_2$.
Similarly, there are at least $(1-4\eps_1)|T|$  vertices $v_3$ in $T$ such that $|N_{G'}(v_3)\cap N_2|\geq d|N_2|/6\geq d^3|S|/216 $.
Continuing in this way we conclude that there are at least
$$|T|\times (1-4\eps _1)^{h-1}|T|^{h-1} \times \frac{1}{h!} \geq (1-4\eps _1)^{h-1}\binom{|T|}{h}$$ 
$h$-sets  $\{v_1,\ldots,v_h\}$ in $\binom{T}{h}$ so that 
$|S\cap N_{G'}(v_1)\cap\ldots \cap N_{G'}(v_h)|\geq d^h|S|/6^h \stackrel{(\ref{hier})}{\geq} d_1 |S|$, as required.
\end{proof}

 Call a good $h$-set $\{v_1,\ldots,v_h\}\in \binom{T}{h}$ \emph{excellent} if in $G_4$ they span a copy of $H'$. 
We now claim that a.a.s.~every subset of $T$ of size at least $\eps_3|T|$ contains an excellent $h$-set.
If $H'$ consists of isolated vertices then every good set is excellent. So we may assume that $H'$ contains an edge.
Recall that $G_4\cong G_{n,p/4}$ and note that $0<d^*(H') \leq d^*(H)$. So  
$$\frac{p}{4}\geq \frac{c}{4}n^{-1/d^*(H)} \geq D \left ( \frac{n}{4Mt} \right )^{-1/d^*(H)}\geq D|T|^{-1/d^*(H')}.$$
The latter inequality holds since $|T|\geq \frac{n}{4Mt}$ by Claim~\ref{claim5}(iii) and Claim~\ref{claim1}(iii).
Let $K'_{|T|}$ be a copy of $K_{|T|}$ on vertex set $T$. By Claim~\ref{newc} there are at most $\eps _2 \binom{|T|}{h}$ copies of $H'$ in $K'_{|T|}$ whose vertex set is a bad $h$-set.
Thus, applying 
Theorem~\ref{jansonthm4} to $G_4[T]\cong G_{|T|,p/4}$,
 implies that a.a.s.~every subset of $T$ of size at least $\eps_3|T|$ contains an excellent $h$-set. 
We may similarly define a good, bad and excellent $h$-set in $\binom{S}{h}$ and 
conclude 
a.a.s.~every subset of $S$ of size at least $\eps_3|S|$ contains an excellent $h$-set.

Let $G^*:=(G'\cup G_4)[S,T]$. We say that a copy $H^*$ of $H$ in 
$G^*$ is an \emph{$S$-copy} if $H^*$ contains precisely one vertex $x$ from $T$ and $H^*- x$ is a copy of $H'$.  Similarly define a \emph{$T$-copy} of $H$ to be a copy of $H$ in $G^*$ intersecting $T$ in a copy of $H'$ and $S$ in one vertex. Set $N:= L/t$;
so $(1-2\eps_ 1)N\leq |S|,|T|\leq (1+\eps)N$ by Claim~\ref{claim5}(iii). The following two claims finish our proof of Theorem~\ref{mainthm} - the first deals with the case where $H$ is a single edge and the second deals with general $H$.
The latter claim is an analogue of  Lemma 3.1 in~\cite{triangle}, and indeed the proof is very similar; we include it for completeness. 

\begin{claim}
If $H=K_2$ is a single edge then a.a.s.~there exists a perfect $H$-tiling in $G^*$.
\end{claim}
\begin{proof}
Without loss of generality we may assume $|S|\leq |T|$. Let $z:=|T|-|S|$ and note that by Claim~\ref{claim5}(iv) we have that $z$ is even. By Theorem~\ref{jansonthm1} we know that in $G_4$ a.a.s.~every set $Q\subset T$ of size $|Q|=\eta n<|T|/2$ contains an edge. Hence we may greedily form a matching $M$ in $G_4[T]$ consisting of $z/2$ edges. Let $T':=T\setminus V(M)$ so that $|T'|=|S|$.

Since $|V(M)|\leq 3\eps_1 |T|$ and $G'[S,T]$ is $(3\eps_1,d/6)$-super-regular, we have that $G'[S,T']$ is $(\eps_2,d/8)$-super-regular.
As $\eps_2 \ll d$, it is easy to see that Hall's condition is satisfied and thus $G'[S,T']$
 contains a perfect matching $M'$. Then $M\cup M'$ forms the desired perfect $H$-tiling in $G^*$.
\end{proof}

\begin{claim}
If $|H|\geq 3$ the following statements hold a.a.s:
\begin{itemize}
\item[(a)] Provided $|S\setminus Q|+|T|+\lfloor \phi\eps_5 N \rfloor$ is divisible by $|H|$, for every $Q\subseteq S$ of size $|Q|=\phi N$ there is an $H$-tiling in $G^*$ which covers every vertex of $(S\cup T)\setminus Q$ and which covers precisely $\lfloor \phi \eps_5 N\rfloor$ vertices of $Q$. Moreover, the same assertion holds if one replaces $T$ with any subset $T'\subseteq T$  where $|T\setminus T'|\leq \eps _5 N$.
\item[(b)] $G^*$ contains a perfect $H$-tiling.
\end{itemize}
\end{claim}
\begin{proof}

For (a), let $z:=\lfloor \phi \eps_5 N \rfloor$, $t:=\lfloor z/h\rfloor$ and $z':=z-ht\in\{0,1,\ldots,h-1\}$. We will construct an $H$-tiling in $G^*$ that covers all of $(S\setminus Q)\cup T$ and precisely $z=ht+z'$ vertices of $Q$. 

Let $T_1'\subseteq T$ consist of all vertices in $T$ with fewer than $d_1|Q|$ neighbours in $Q$. Note $G'[S,T]$ is $(\eps _2,d_1)$-super-regular so we have that $|T'_1|\leq 2\eps _2 N$. Form $T_1$ by adding at most $h$ arbitrarily selected vertices from $T \setminus T_1'$ to $T_1'$ so that $|T\setminus T_1|-t$ is divisible by $|H|$. Since $G'[S,T]$ 
is $(\eps_2,d_1)$-super-regular, every vertex of $T_1$ has at least $d_1|S|-|Q|\geq \frac{d_1N}{2} > h|T_1|+2\eps_3 N$ neighbours in $S\setminus Q$. 
As every subset of $S$ of size $2\eps_3 N$ contains an excellent $h$-set, we may greedily form an $H$-tiling $\mathcal{T}_1$ of $S$-copies in $G^*$ of size $|T_1|$ which covers every vertex of $T_1$ and does not use any vertex from $Q$.

 We now select uniformly at random a subset $T_2\subseteq T\setminus T_1$ of size $|T_2|=t$. Since every vertex in $S$ has at least $d_1|T|-|T_1|\geq \frac{d_1N}{2}$ neighbours in $T\setminus T_1$, Chernoff's inequality for the hypergeometric distribution
 implies that, with probability $1-o(1)$, every vertex of $S$ has at least $\frac{\phi \eps_5 d_1}{5h}N$ neighbours in $T_2$. Fix a choice of $T_2$ for which this event occurs. Let $Q'$ be an arbitrarily selected subset of $Q$ of size $z'$, so that $|Q'|\leq h-1$ and let $S' := (S\setminus (Q\cup V(\mathcal{T}_1)))\cup Q'$ and $T' := T\setminus (T_1\cup T_2)$. Recall that, by assumption, $|S\setminus Q| + |T| + z$ 
is divisible by $|H|$, so $$|S'|+|T'|=|S\setminus Q|+z'+|T|-|T_2|-|V(\mathcal{T}_1)|=(|S\setminus Q| + |T| + z) - (|H|t + |V(\mathcal{T}_1)|)$$is divisible by $|H|$. Since $|T'|$ is divisible by $|H|$ by our selection of $T_1$ and $T_2$, it follows that $|S'|$ is divisible by $|H|$ as well.

Let $t_2:=\lfloor\frac{\phi\eps_5d_1}{10h^2}N\rfloor$, $a:=\frac{h}{(h-1)|H|}|S'|-\frac{1}{(h-1)|H|}|T'|$ and $b:=\frac{h}{(h-1)|H|}|T'|-\frac{1}{(h-1)|H|}|S'|-t_2$. Note that since $|H|\geq 3$ both $a$ and $b$ are positive. 
\begin{subclaim}
There is an $H$-tiling $\mathcal{T}_2$ in $G^*[S'\cup T']$ 
which consists of $a$ $S$-copies and $b$ $T$-copies. In particular, 
 $S'' := S'\setminus V(\mathcal{T}_2)$ and $T'' := T'\setminus V(\mathcal{T}_2)$ have sizes precisely $|S''|=|S'|-(ha+b)=t_2$ and $|T''|=|T'|-(a+hb)=ht_2$. 
\end{subclaim}
\proof
Note $G'[S',T']$ is $(2\eps_2,d_1/2)$-super-regular.
The copies of $H$ in $\mathcal T_2$ may be chosen greedily. Indeed, suppose that we have already chosen an $H$-tiling $\mathcal{T}$ in $G^*[S',T']$ consisting of at most $a$ $S$-copies and $b$ $T$-copies, then $\mathcal{T}$ covers at most $ha+b$ vertices of $S$, and at most $a+hb$ vertices of $T$. Taking $S^* := S' \setminus V(\mathcal{T})$ and $T^* := T'\setminus V(\mathcal{T})$, we find that $|S^*|\geq t_2 \geq \eps_4 N$ and $|T^*|\geq ht_2 \geq \eps _4 N$. 
Since $G'[S',T']$ is $(2\eps_2,d_1/2)$-super-regular, it follows that there is some vertex $x\in S^*$ having at least $d_1\eps_4 N/2> 2\eps_3 N$ neighbours in $T^*$. Since every set of size at least $2\eps_3 N$ in $T^*$ contains an excellent $h$-set this gives a $T$-copy  which can be added to $\mathcal{T}$. The same argument with the roles of $S^*$ and $T^*$ reversed yields instead an $S$-copy which may be added to $\mathcal{T}$.  This proves the claim.
\endproof

Since by the choice of $T_2$ each vertex of $S''$ has at least $\frac{\phi \eps_5 d_1}{5h}N>h|S''|+2\eps_3 N$ neighbours in $T_2$, we may greedily form an $H$-tiling $\mathcal{T}_3$ in $G^*[S''\cup T_2]$ consisting of $t_2$ $T$-copies which covers every vertex of $S''$ and which covers precisely $ht_2$ vertices of $T_2$. At this point we have obtained an $H$-tiling $\mathcal{T}_1\cup \mathcal{T}_2\cup\mathcal{T}_3$ in $G^*$ which covers every vertex of $S$ except for those in $Q\setminus Q'$ and every vertex of $T$ except for the precisely $ht_2$ vertices in $T''$ and the precisely $t-ht_2$ vertices in $T_2\setminus V(\mathcal{T}_3)$. Therefore, in total, precisely $t$ vertices of $T$ remain uncovered, each of which has at least $d_1|Q|-|Q'|>ht+2\eps_3 N$ neighbours in $Q\setminus Q'$ by the choice of $T_1$. We may therefore greedily form an $H$-tiling $\mathcal{T}_4$ of $S$-copies in $G^*$ which covers all the remaining uncovered vertices in $T$ and precisely $ht$ vertices of $Q\setminus Q'$. Then $\mathcal{T}_1\cup \mathcal{T}_2\cup \mathcal{T}_3\cup \mathcal{T}_4$ is the desired $H$-tiling.

The moreover part of (a) follows by observing that since $G'[S,T]$ is $(3\eps_1,d/6)$-super-regular,  $G'[S,T' ]$ is $(\eps_2,d_1)$-super-regular for any subset $T'\subseteq T$ where $|T\setminus T'|\leq \eps _5 N$. Thus, the precise argument above, with $T'$ playing the role of $T$, implies we have the desired $H$-tiling.

\smallskip

For (b), we may assume without loss of generality that $|T|\geq |S|$. Since every set of at least $2\eps_3 N$ vertices of $T$ contains an excellent $h$-set we may greedily form an $H'$-tiling $M$ of size at least $(|T|-2\eps_3 N)/h\geq N/(h+1)$ in $G_4[T]$ such that the vertex set of each of these copies of $H'$ are excellent. Fix such an $H'$-tiling $M$, and form an auxiliary bipartite graph $K$ with vertex classes $S$ and $M$ where $a\in S$ and $e=(x_1x_2\ldots x_h)\in M$ are adjacent if and only if $a\in N_{G^*}(x_1)\cap\ldots\cap N_{G^*}(x_h)$. Note that for every $H'$-copy $e=(x_1x_2\ldots x_h)\in M$ we have that 
$$d_{K}(e)=|N_{G^*}(x_1)\cap\ldots\cap N_{G^*}(x_h)\cap S|\geq d_1 |S|$$ by the definition of an excellent $h$-set, so $K$ has density at least $d_1$. We now apply the following lemma from~\cite{triangle}.

\begin{lemma}[\cite{triangle}, Lemma 2.6]
Suppose that $1/n \ll \phi ' \ll \eps ' \ll d'\ll 1/h'$.
Let $F$ be a bipartite graph with vertex classes $A$ and $B$ such that $n/h'\leq |A|,|B|\leq n$ and $d_F(A,B)\geq d'$. Then there exist subsets $X\subseteq A$ and $Y\subseteq B$ of sizes $|X|=\phi' n$ and $|Y|=(1-\eps ') \phi ' n$ such that $F[X',Y]$ contains a perfect matching for every subset $X'\subseteq X$ with $|X'|=|Y|$. 
\end{lemma}
We remark that this lemma is not precisely as stated in~\cite{triangle}, but this version can easily be deduced from the original.

Hence we may choose subsets $X\subseteq S$ and $M'\subseteq M$ such that $|X|=\phi N$, $|M'|=(1-\eps_5)\phi N$ and such that $K[X',M']$ contains a perfect matching for every subset $X'\subseteq X$ with $|X'|=|M'|$. Let $T' := T\setminus V(M')$. Then, since $\phi \ll \eps_5$  we may apply (a) to $G^*[S\cup T']$ with  $X$ playing the role of $Q$  to obtain an $H$-tiling $\mathcal{T}_1$ in $G^*$ which covers every vertex of $G^*$ except for the vertices of $V(M')$ and precisely $(1-\eps_5)\phi N$ vertices of $X$. So, taking $X'$ to be the vertices of $X$ not covered by $\mathcal{T}_1$, we have $|X'|=|M'|$. By the choice of $X$ and $M'$ it follows that $K[X',M']$ contains a perfect matching, which corresponds to a perfect $H$-tiling $\mathcal{T}_2$ in $G^*[X'\cup V(M')]$. This gives a perfect $H$-tiling $\mathcal{T}_1\cup \mathcal{T}_2$ in $G^*$.

\end{proof}

\section*{Acknowledgment}
The authors are grateful to the BRIDGE strategic alliance between the University of Birmingham and the University of Illinois at Urbana-Champaign. This research was conducted as part of the `Building Bridges in Mathematics' BRIDGE Seed Fund project.

The authors are also grateful to the referee for their careful and helpful review.

\end{document}